\documentclass[letterpaper,10pt]{amsart}

\usepackage[all]{xy}                        %

\CompileMatrices                            % Faster

\UseTips                                    % Use

\input xypic
\usepackage[bookmarks=true]{hyperref}       % Hyperref
\usepackage{amssymb,latexsym,amsmath,amscd}
\usepackage{xspace}
\usepackage{color}
\usepackage{graphicx}
\reversemarginpar

\theoremstyle{plain}
\newtheorem{theorem}{Theorem}[section]
\newtheorem*{theorem*}{Theorem}
\newtheorem{proposition}[theorem]{Proposition}
\newtheorem{corollary}[theorem]{Corollary}
\newtheorem{lemma}[theorem]{Lemma}

\theoremstyle{definition}
\newtheorem{definition}[theorem]{Definition}

\newtheorem{remark}[theorem]{Remark}

\newcommand{\enm}[1]{\ensuremath{#1}}          %
\newcommand{\op}[1]{\operatorname{#1}}
\newcommand{\cal}[1]{\mathcal{#1}}

\newcommand{\CC}{\enm{\mathbb{C}}}

\newcommand{\NN}{\enm{\mathbb{N}}}

\newcommand{\ZZ}{\enm{\mathbb{Z}}}
\newcommand{\FF}{\enm{\mathbb{F}}}

\newcommand{\PP}{\enm{\mathbb{P}}}

\newcommand{\Aa}{\enm{\cal{A}}}
\newcommand{\Bb}{\enm{\cal{B}}}

\newcommand{\Ee}{\enm{\cal{E}}}
\newcommand{\Ff}{\enm{\cal{F}}}
\newcommand{\Gg}{\enm{\cal{G}}}
\newcommand{\Hh}{\enm{\cal{H}}}
\newcommand{\Ii}{\enm{\cal{I}}}

\newcommand{\Ll}{\enm{\cal{L}}}

\newcommand{\Nn}{\enm{\cal{N}}}
\newcommand{\Oo}{\enm{\cal{O}}}

\newcommand{\Rr}{\enm{\cal{R}}}

\renewcommand{\phi}{\varphi}
\renewcommand{\theta}{\vartheta}
\renewcommand{\epsilon}{\varepsilon}

\newcommand{\Pic}{\op{Pic}}

\newcommand{\Hom}{\op{Hom}}

\newcommand{\End}{\op{End}}

\newcommand{\rk}{\op{rank}}

\newcommand{\id}{\op{id}}

\renewcommand{\to}[1][]{\xrightarrow{\ #1\ }}

\newcommand{\old}[1]{}

\begin{document}

%\layout
\title[A note on co-Higgs bundles]{A note on co-Higgs bundles}
\author{Edoardo Ballico and Sukmoon Huh}
\address{Universit\`a di Trento, 38123 Povo (TN), Italy}
\email{edoardo.ballico@unitn.it}
\address{Sungkyunkwan University, 300 Cheoncheon-dong, Suwon 440-746, Korea}
\email{sukmoonh@skku.edu}
%\address{Politecnico di Torino, Corso Duca degli Abruzzi 24, 10129 Torino, Italy}
%\email{francesco.malaspina@polito.it}
\thanks{The first author is partially supported by MIUR and GNSAGA of INDAM (Italy). The second author is supported by Basic Science Research Program 2015-037157 through NRF funded by MEST}


\begin{abstract}
We show that for any ample line bundle on a smooth complex projective variety with nonnegative Kodaira dimension, the semistability of co-Higgs bundles of implies the semistability of bundles. Then we investigate the criterion for surface $X$ to have $H^0(T_X)=H^0(S^2T_X)=0$, which implies that any co-Higgs structure of rank two is nilpotent.
\end{abstract}

\maketitle
%\tableofcontents

\section{Introduction}
A co-Higgs bundle on a smooth complex projective variety $X$ is a pair $(\Ee, \Phi)$, where $\Ee$ is a vector bundle on $X$ and $\Phi$, the co-Higgs field, is a morphism $\Ee \rightarrow \Ee \otimes T_X$, satisfying the integrability condition $\Phi \wedge \Phi=0$. It is introduced and developed in \cite{Hi, Gual} as a generalized vector bundle over $X$, considered as a generalized complex manifold. 

There have been recent interests on the classification of stable co-Higgs bundles on lower dimensional varieties. In \cite{R1, Rayan}, Rayan describe the moduli spaces of stable co-Higgs bundles of rank two both on the projective line and on the projective plane, and show several non-existence results of them over varieties with nonnegative Kodaira dimension $\kappa(X)\ge 0$. The main philosophy is that the existence of stable co-Higgs bundles determine the position of $X$ toward negative direction in the Kodaira spectrum. Indeed it is shown in \cite{Correa} that if $\dim (X)=2$, the existence of semistable co-Higgs bundle $(\Ee, \Phi)$ of rank two with $\Phi$ nilpotent, implies that $X$ is either uniruled, a torus, or a properly elliptic surface, up to finite \'etale cover. Colmenares also describes the moduli space of semistable co-Higgs bundles of rank two on Hirzebruch surfaces in \cite{VC, VC1}. 

Our goal in this article is twofold. We first investigate the relationship between semistability of $(\Ee, \Phi)$ and semistability of $\Ee$;

\begin{theorem}
With respect to any ample line bundle on a smooth projective variety with nonnegative Kodaira dimension, if a co-Higgs bundle $(\Ee, \Phi)$ is semistable, then $\Ee$ is also semistable. 
\end{theorem}

In fact, when we disregard the condition on Kodaira dimension, we also get a similar statement on stabilities under the assumption that the strict order of  instability is low (see Proposition \ref{t2}). The results give a way to study moduli of semistable co-Higgs bundles, with a base on the study of moduli of semistable bundles. Indeed, Proposition \ref{tt5} asserts that any semistable co-Higgs bundle of rank two over a surface of general type has a trivial co-Higgs field, and so the semistability of co-Higgs bundles of rank two is equivalent to the semistability of bundles. 

Then we pay our attention to the case when $X$ is a surface and the rank of co-Higgs bundles is two. Under vanishing of $H^0(T_X)$, the existence of unstable co-Higgs bundle of rank two implies that the co-Higgs field is nilpotent (see Lemma \ref{tt6}), while any non-trivial global tangent vector field suggests an example of strictly semistable co-Higgs bundle of arbitrary rank with injective co-Higgs field (see Lemma \ref{oo1}). It motivates to seek for additional conditions to assure that a semistable co-Higgs bundle has a nilpotent co-Higgs field. By the argument in \cite[Theorem 7.1, page 148--149]{R1}, the vanishing condition $H^0(T_X)=H^0(S^2T_X)=0$ implies that co-Higgs fields are nilpotent. Indeed, any surface can achieve this vanishing after a finite number of blow-ups. 

\begin{theorem}
For a surface $X$, there exists a surface $X'$ and a birational morphism $u: X'\rightarrow X$ with the following property. If $v: X''\rightarrow X'$ is any birational morphism, then every rank two co-Higgs field on $X''$ is nilpotent. 
\end{theorem}

Since, due to the Enriques-Kodaira classification, we have a list of minimal model $Y$ of $X$ together with information on $H^0(T_Y)$ and $H^0(S^2T_Y)$, we are able to find several classes of surfaces with the prescribed vanishing: blow-ups of the projective plane, Hirzebruch surfaces, abelian surfaces and a type of properly elliptic surfaces.

Let us summarize here the structure of this article. In section $2$ we introduce the definition of co-Higgs bundles and a notion of semistability with respect to a fixed ample line bundle. Then we discuss the relationship between semistability of co-Higgs bundles and semistability of bundles, together with nilpotent co-Higgs fields. Main ingredients are the generic nefness of the cotangent bundle of non-uniruled $X$ and certain extensions that nilpotent co-Higgs fields induce. In section $3$, we study the extensions above  via holomorphic foliation in \cite{SB} to characterize the base variety $X$ in Kodaira spectrum, when it admits a semistable co-Higgs bundle of rank two whose bundle factor is not semistable with non-zero co-Higgs field. In section $4$, we mainly work on the criterion for vanishing $H^0(T_X)$ and $H^0(S^2T_X)$ and it is observed that the vanishing can be achieved by blow-ups sufficiently many times.

%%%%%%%%%%%%%%%%%%%%%%%%%%%

\section{Preliminaries}
Throughout the article our base field is the field $\CC$ of complex numbers. We will always assume that $X$ is a smooth projective variety with a fixed very ample line bundle $\Oo_X(1)$ and the tangent bundle $T_X$. For a coherent sheaf $\Ee$ on a projective scheme $X$, we denote $\Ee \otimes \Oo_X(t)$ by $\Ee(t)$ for $t\in \ZZ$. The dimension of cohomology group $H^i(X, \Ee)$ is denoted by $h^i(X,\Ee)$ and we will skip $X$ in the notation, if there is no confusion. 

\begin{lemma}\label{t1}
$\Omega _X^1(2)$ is globally generated.
\end{lemma}

\begin{proof}
Consider $X$ as a subvariety of a projective space $\PP^r$ such that the embedding is given by a complete linear system $|\Oo _X(1)|$. Then we have an exact sequence
$$0 \to N_{X|\PP^r}^\vee (2) \to {\Omega _{\PP^r}^1(2)}_{\vert_X} \to \Omega _X^1(2) \to 0,$$
where $N_{X|\PP^r}$ is the normal bundle of $X$ in $\PP^r$. Note that $\Omega _{\PP ^r}^1(2)$ is $0$-regular by the Bott formula. Thus it is globally generated and so is $\Omega_X^1 (2)$. 
\end{proof}
In particular, we have $h^0(T_X(-2))=0$. Now we give the definition of co-Higgs bundle, which is the main object of this article. 

\begin{definition}
A {\it{co-Higgs}} bundle on $X$ is a pair $(\Ee, \Phi)$ where $\Ee$ is a vector bundle on $X$ and $\Phi \in H^0(\mathcal{E}nd (\Ee) \otimes T_X)$ for which $\Phi\wedge \Phi=0$ as an element of $H^0(\mathcal{E}nd(\Ee) \otimes \wedge^2 T_X)$. Here $\Phi$ is called the {\it{co-Higgs field}} of $(\Ee, \Phi)$ and the condition $\Phi \wedge \Phi=0$ is called the {\it integrability}. 
\end{definition}

We may constrain ourselves to the co-Higgs bundles $(\Ee, \Phi)$ with trace zero, i.e. $\Phi$ is contained in $H^0(\mathcal{E}nd_0 (\Ee) \otimes T_X)$, because each co-Higgs field $\Phi$ can be decomposed as $(\phi_1, \phi_2)$, where $\phi_1\in H^0(\mathcal{E}nd_0(\Ee)\otimes T_X)$ and $\phi_2 \in H^0(T_X)$ due to the splitness of the trace map sequence: $0\rightarrow \mathcal{E}nd_0(\Ee) \rightarrow \mathcal{E}nd(\Ee) \stackrel{\mathrm{tr}}{\rightarrow} \Oo_X \rightarrow 0$. In particular, if $H^0(T_X)=0$, then every co-Higgs field has trace-zero.

%\begin{remark}
%Assuming $H^0(-K_X)=0$, we get that $H^0(T_X)=0$ and so every co-Higgs field is trace-free, or equivalently, no co-Higgs bundle of rank one has a non-trivial co-Higgs field. 
%\end{remark}

\begin{definition}\label{ss1}
A co-Higgs bundle $(\Ee, \Phi)$ is {\it semistable} (resp. {\it stable}) if
\begin{align*}
\frac{\deg \Ff}{\rk \Ff} \leq~ (\text{resp.} <)~\frac{\deg \Ee}{\rk \Ee}
\end{align*}
for every coherent subsheaf $0\subsetneq \Ff \subsetneq \Ee$ with $\Phi(\Ff) \subset \Ff \otimes T_X$. 
\end{definition}

The following observation shows why to consider (non-)existence results for co-Higgs bundles $(\Ee, \Phi)$ with $\Phi \ne 0$ one usually assume that $(\Ee ,\Phi )$ is semistable.

\begin{remark}
Take a positive integer $k$ such that $T_X(k)$ is globally generated and choose a non-zero section $\sigma \in H^0(T_X(k))$. For an integer $r\ge 2$, set $\Ee:= \Oo _X\oplus \Oo _X(k)^{\oplus (r-1)}$. If we define a map $\Phi : \Ee \rightarrow \Ee \otimes T_X$ to send the factor $\Oo _X$ to one of the factor, $T_X(k)$, of $\Ee \otimes T_X$ using $\sigma$ and send $\Oo _X(k)^{\oplus (r-1)}$ onto $0$. By construction we have $\Phi \circ \Phi = 0$ and so $\Phi \wedge \Phi =0$.
\end{remark}

\begin{lemma}\label{tt6}
Assume that $H^0(T_X)=0$. If a co-Higgs bundle $(\Ee,\Phi)$ of rank two on $X$ is not stable, then $\Phi$ is nilpotent.
\end{lemma}

\begin{proof}
We may assume $\Phi \ne 0$. Since $(\Ee ,\Phi )$ is not stable, there is a line subbundle $\Ll\subset \Ee$ such that $\Phi (\Ll) \subset \Ll\otimes T_X$ and $\deg \Ll \ge \deg \Rr/2$, where $\Rr:= \det (\Ee )$. The saturation $\Ll'$ of $\Ll$ in $\Ee$ satisfies $\Phi (\Ll')\subset \Ll'\otimes T_X$, because $\Ll'\otimes T_X$ is the saturation of $\Ll\otimes T_X$ in $\Ee\otimes T_X$. Thus we may assume that $\Ll$ is saturated in $\Ee$ with an exact sequence
\begin{equation}\label{eqa11}
0 \to \Ll \to \Ee \to \Ii_Z\otimes \Rr\otimes \Ll^\vee \to 0.
\end{equation}
Due to vanishing $H^0(T_X)=0$, the inclusion $\Phi (\Ll) \subset \Ll\otimes T_X$ implies $\Phi (\Ll) =0$. So $\Phi$ induces a map $u: \Ii _Z\otimes \Rr \otimes \Ll ^\vee \rightarrow \Ee \otimes T_X$. Composing $u$ with the map $\Ee \otimes T_X \rightarrow \Ii _Z\otimes \Rr \otimes T_X$ induced by (\ref{eqa11}) we get a map $v: \Ii _Z\otimes \Rr\otimes \Ll^\vee \rightarrow  \Ii _Z\otimes \Rr\otimes \Ll^\vee \otimes T_X$ . Again from $H^0(T_X)=0$ we get $v=0$, i.e.
$\mathrm{Im}(\Phi )\subset \Ll \otimes T_X$ and so $\Phi ^2=0$.
\end{proof}

\begin{remark}\label{oo1}
Assume $H^0(T_X)\ne 0$ with a fixed non-zero section $\sigma \in H^0(T_X)$. Fix an open subset $U\subset X$, where $T_X$ is trivial and let $\partial _1$, $\partial _2$ be a basis of $T_U\cong \Oo _U^{\oplus 2}$. If $\Phi : \Ll \rightarrow \Ll\otimes T_X$ is the map induced by $\sigma$ for a line bundle $\Ll$, then we may write $\Phi _{|U} = f_1(z)\partial _1+f_2(z)\partial _2$ with $f_i(z)\in H^0(\Oo _U)$. We have $\Phi \wedge \Phi  = (f_1f_2-f_2f_1)\partial _1\wedge \partial _2=0$ (see \cite[Remark 2.29]{VC} and \cite{VC1}) and so $(\Ll, \Phi)$ is a co-Higgs bundle. For a fixed integer $r\ge 1$, let $(\Ee, \Phi)$ denote the co-Higgs bundle which is the direct sum of $r$ copies of $(\Ll^{\oplus r}, \sigma)$. Then $\Phi$ is injective and so it is not nilpotent. Notice that $(\Ee ,\Phi )$ is strictly semistable for any polarization on $X$.
\end{remark}

\begin{remark}
Let $\Ee$ be a vector bundle of rank two on $X$, which is not semistable with respect to $\Oo _X(1)$ with the following destabilizing sequence
\begin{equation}\label{eqaa1}
0 \to \Ll \to \Ee \to \Ii _Z\otimes \Aa \to 0.
\end{equation}
Assume the existence of a co-Higgs field $\Phi$ on $X$ such that $(\Ee ,\Phi )$ is semistable. Clearly we get $\Phi \ne 0$. Since $h^0(T_X(-2)) =0$ by Lemma \ref{t1}, we have $h^0(\Ii _Z\otimes \Aa\otimes(\Ll^\vee)^{\otimes 2}(2)) >0$. Later in Lemma \ref{lem4} we see that $(\Ee ,\Phi )$ is associated to a foliation by rational curves. 
\end{remark}

For a vector bundle $\Ee$ of rank two on $X$, let us choose $\Rr \in \mbox{Pic}(X)$ so that $\Ff := \Ee \otimes \Rr$ has a non-zero section with a $2$-codimensional zero locus $Z$ and an exact sequence
\begin{equation}\label{eqar2}
0 \to \Oo _X \to \Ff \to \Ii _Z \otimes \Ll\to 0
\end{equation}
with $\Ll := \det (\Ff )=\det (\Ee )\otimes \Rr^{\otimes 2}\in \mbox{Pic}(X)$. Note that there is an obvious bijection between the co-Higgs fields (and the co-Higgs integrable structures) of $\Ee$ and $\Ff$. The {\it {strict order of instability}} of $\Ee$ is the maximal integer $k$ such that $h^0(\Ll(-k)) >0$ for all possible sequences (\ref{eqar2}) and $\Rr$, denoted by $\mathrm{ord}^{\mathrm{in}}(\Ee)$. In particular, the strict order of instability is invariant under the twist by line bundles. 

% Fix an integer $k>0$ and $h^0(L(-k)) >0$. Then $\Ee$ and $\Ff$ are $\Oo _X$-unstable of order at least $k$. Obviously $\Ff$ is not $\Oo _X(1)$-semistable and hence (\ref{eqa2}) is the only $\Oo _X(1)$-destabilizing sequence for $\Ee$ and $\Ff$. In this case (\ref{eqa2}) is uniquely determined by $\Ee$ and $\Oo _X(1)$. 

\begin{proposition}\label{t2}
Let $\Ee$ be an unstable bundle of rank two on $X$ with $\mathrm{ord}^{\mathrm{in}}(\Ee)\le -3$. Then $(\Ee, \Phi)$ is not stable for any co-Higgs field $\Phi$.  
\end{proposition}

\begin{proof}
Take a maximal destabilizing line subbundle $\Rr$ of $\Ee$ and consider (\ref{eqar2}). It is enough to show that $\Phi(\Rr^\vee )\subset \Rr^\vee \otimes T_X$ for every nonzero co-Higgs field $\Phi: \Ee \rightarrow \Ee \otimes T_X$. Then $\Rr^\vee$ would destabilize the co-Higgs bundle $(\Ee ,\Phi)$. Note that $\mathrm{ord}^{\mathrm{in}}(\Ee)\le 1$ implies that $H^0(\Ll(2))=0$. With notations above, let $\Phi': \Ff \rightarrow \Ff \otimes T_X$ be the nonzero map induced by $\Phi$. Then we need to prove that $\Phi'(\Oo _X) \subset \Oo_X\otimes T_X$, i.e. that the induced map $\Oo _X \rightarrow  \Ii _Z\otimes \Ll\otimes T_X$ is a zero map. Since $H^0(\Ll (2))=0$, so we have $H^0(\Ii_Z \otimes \Ll (2))=0$. Since $\Omega_X^1 (2)$ is globally generated by Lemma \ref{t1}, we have $H^0(\Ii_Z \otimes \Ll \otimes T_X)=0$. 
\end{proof}

Proposition \ref{t2} is sharp, as shown in \cite[Remarks 3.1, 5.1 and Proposition 5.1]{R1} and \cite[Proposition 5.1]{Rayan}. When $\kappa (X)$ is non-negative, Proposition \ref{t2} may be improved in the following way, applying Definition \ref{ss1} with $\Hh$ instead of $\Oo_X(1)$. 

%\begin{proposition}\label{t2.00}
%Let $(\Ee,\Phi)$ is a co-Higgs bundle of rank two on a smooth projective variety $X$ of arbitrary dimension with $\kappa(X)\ge 0$. For any ample line bundle $\Hh$ on $X$, if $\Ee$ is not $\Hh$-semistable, then $(\Ee, \Phi)$ is not also $\Hh$-semistable. 
%\end{proposition}

%\begin{proof}
%Set $n:=\dim (X)$ and let $\Ll\subset \Ee$ be a line bundle with maximal value of $\Ll \cdot \Hh$. By assumption $\Ll$ is a saturated subsheaf of rank one and we have $2\Ll \cdot \Hh > \det (\Ee)\cdot \Hh^{n-1}$. Then it is sufficient to prove that $\Phi (\Ll )\subset \Ll \otimes T_X$. Assume not and then by (\ref{eqar2}) $\Phi _{|\Ll}$ induces a non-zero map $u: \Ll \rightarrow \det (\Ee)\otimes \Ll ^\vee \otimes T_X$ and so a non-zero map $v: \Omega ^1_X \rightarrow \det (\Ee )\otimes (\Ll^{\vee})^{ \otimes 2}$. Fix an integer $m\gg 0$ so that $\Hh^{\otimes m}$ is very ample, and take a general complete intersection $C\subset X$ of $n-1$ elements of $|\Hh^{\otimes m}|$. By \cite{Miya}, \cite{SB} or \cite[Theorem 4.1]{Pet}, $(\Omega _X^1)_{|C}$ is nef and so any rank one quotient of $(\Omega _X^1)_{|C}$ has non-negative degree. Thus we have $2\Ll\cdot \Hh^{n-1} \le \det (\Ee )\cdot \Hh^{n-1}$, a contradiction. 
%\end{proof}

\begin{theorem}\label{t2.10}
Let $X$ be a smooth projective variety with $\kappa (X) \ge 0$. For any ample line bundle $\Hh$ on $X$, If $\Ee$ is a torsion-free sheaf which is not $\Hh$-semistable, then no co-Higgs sheaf $(\Ee ,\Phi)$ is $\Hh$-semistable.
\end{theorem}

\begin{proof}
Set $n:= \dim (X)$ and $r:= \mathrm{rank}(\Ee)$, and let $0=\Ee_0\subset\Ee _1\subset \cdots \subset \Ee _k=\Ee$ be the Harder-Narasimhan filtration of $\Ee$ (see \cite{HL}). Since $\Ee$ is not semistable, we have $r\ge 2$ and $k\ge 2$. We may also assume $n\ge 2$, since the case $n=1$ is known in \cite{R1, Rayan}. If $\Phi (\Ee _1)\subset \Ee _1\otimes T_X$, then $(\Ee ,\Phi )$ is not semistable and so we may also assume that $\Phi (\Ee _1)\nsubseteq \Ee _1\otimes T_X$. Therefore $\Phi _{|\Ee _1}$ induces a non-zero map $u: \Ee _1\rightarrow(\Ee/\Ee _1)\otimes T_X$. Set $\Aa := \mathrm{Im}(u)$ and then it is torsion-free. Let $0=\Aa_0\subset\Aa_1\subset \cdots \subset \Aa _s =\Aa$ be the Harder-Narasimhan filtration of $\Aa$. Two of the properties of the Harder-Narasimhan filtration say that 
\begin{itemize}
\item $\Ee _i/\Ee _1$ is torsion free for $i=2,\dots ,k$ and 
\item $0\subset \Ee _2/\Ee _1\subset \cdots \subset \Ee _/\Ee _1$ is the Harder-Narasimhan filtration of $\Ee /\Ee _1$. 
\end{itemize}
Since $\Aa$ is a quotient of the semistable sheaf $\Ee _1$, the normalized Hilbert polynomial of $\Aa _1$ and each subquotients $\Aa _{i+1}/\Aa _i$ with $0\le i<s$, is at least the one of $\Ee _1$ and so bigger than the ones of $\Ee _i/\Ee _{i-1}$ for $i=2,\dots ,k$.

Fix an integer $m\gg 0$ so that $\Hh^{\otimes m}$ is very ample and take a general complete intersection $C\subset X$ of $n-1$ elements of $|\Hh^{\otimes m}|$. By \cite{Miya}, \cite{SB} or \cite[Theorem 4.1]{Pet}, ${\Omega _X^1}_{|C}$ is nef. Since $\Ee_i$, $\Ee_i/\Ee_{i-1}$, $\Aa_j$ and $\Aa_j/\Aa_{j-1}$ are all torsion-free for $i=1,\ldots, k$ and $j=1, \ldots,s$, they are locally free outside a finite union of two-codimensional subvarieties of $X$. Thus for a general $C$ all these sheaves are locally free in a neighborhood of $C$. Since $u$ is non-zero, we get a non-zero map $v: \Aa _{|C} \rightarrow (\Ee /\Ee _1)\otimes T_{X|C}$ and so a non-zero map $w: {\Omega _X^1}_{|C} \rightarrow \Aa ^\vee _{|C}\otimes {\Ee /\Ee _1}_{|C}$.  By \cite{MR} the restrictions $\Aa _
 {1|C}$, ${\Ee _i/\Ee _{i-1}}_{|C}$, and ${\Aa _j/\Aa _{j-1}}_{|C}$ for $i=2,\dots ,k$ and $j=2,\ldots, s$, are all semistable, i.e. the restriction to $C$ of the Harder-Narasimhan filtrations of $\Ee$ and $\Aa$ are the Harder-Narasimhan filtrations of the vector bundles $\Ee _{|C}$ and $\Aa _{|C}$, respectively. Note that 
\begin{itemize}
\item the slope of the tensor product of two vector bundles on $C$ is the sum of their slopes and
\item the tensor product of two semistable bundles on $C$ is semistable (see \cite[Corollary 6.4.14]{Lazarsfeld}).
\end{itemize}
Thus all the Harder-Narasimhan subquotients of $\Aa ^\vee _{|C}\otimes (\Ee /\Ee _1)_{|C}$ are semistable vector bundles of negative degree. Let $0=\Ff_0\subset \Ff _1\subset \cdots \subset \Ff _h$ be the Harder-Narasishan filtration of $\Aa ^\vee _{|C}\otimes (\Ee /\Ee _1)_{|C}$ and let $l$ be the minimal integer such that $\mathrm{Im}(w) \subseteq \Ff _l$. Since ${\Omega ^1_X}_{|C}$ is nef, any quotient of ${\Omega ^1_X}_{|C}$ has non-negative degree. If $l=1$, every non-zero subsheaf of $\Ff _1$ has negative degree, since $\Ff _1$ is semistable with $\deg (\Ff_1)<0$, a contradiction. If $l>1$, then we get a contradiction taking the composition of $w$ with the surjection $\Ff _l\rightarrow \Ff _l/\Ff _{l-1}$.
\end{proof}

%\begin{remark}
%Let $X$ be a smooth quadric surface and set $\Ll=\Oo_Q(a,b)\in \mathrm{Pic}(X)$. Since $T_X \cong \Oo _{Q}(2,0)\oplus \Oo _{Q}(0,2)$, we may assume that $a,b<0$ to get the instability of any co-Higgs field on $\Ee$ as in Lemma \ref{t2}.
%\end{remark}

As shown in \cite{Rayan}, many unstable vector bundles such as decomposable ones may give stable co-Higgs bundles. The next observation and Lemma \ref{t1} shows in particular that we cannot increase too much the stability. In case of stable bundles this phenomenon may be measured by the following observation.

\begin{remark}\label{t3}
In (\ref{eqar2}) assume $h^0(\Ii _Z\otimes \Ll\otimes T_X) =0$, e.g. $h^0(\Ii _Z\otimes \Ll(2))=0$, and then every co-Higgs field on $\Ff$ preserves $\Oo _X$. Thus for fixed $\Ll$ and ``~large~'' $Z$ we have several examples in which any co-Higgs field cannot be more stable than the order of stability
represented by the non-zero map $\Oo _X\rightarrow \Ff$ in (\ref{eqar2}).
\end{remark}

Note that for any line bundle $\Aa\in \Pic (X)$, a co-Higgs bundle $(\Ee ,\Phi )$ is semistable or stable if and only if $(\Ee \otimes \Aa ,\Phi _\Aa)$ has the same property, where $\Phi _\Aa$ is induced by $\Phi$ by tensoring with $\Aa$. So in case of rank two, we may reduce many problems to the case in which $\det (\Ee )$ is in a prescribed class of $\mathrm{Pic}(X)/2\mathrm{Pic}(X)$. For many surfaces we have  $2\mathrm{Pic}(X)\subsetneq \mathrm{Pic}(X)$ and so we cannot reduce all problems to the case in which $\det (\Ee )\cong \Oo _X$, e.g. if $X$ is a surface of general type, which is not minimal. This explains why we extend \cite[Theorem 7.1]{R1} in the following way; we also exclude the strictly semistable case.

\begin{proposition}\label{tt5}
Let $X$ be a surface of general type. Then there is no semistable co-Higgs bundle $(\Ee,\Phi)$ of rank two with $\Phi \ne 0$.
\end{proposition}

\begin{proof}
Since $X$ is of general type, we have $H^0(T_X)=0$ and so $\Phi$ has trace zero. The integrability $\Phi \wedge \Phi =0$ gives a map $\Phi \circ \Phi : \Ee \rightarrow \Ee \otimes S^2T_X$. As in the proof of \cite[Theorem 7.1, page 148--149]{R1}, we get $\Ll := \mathrm{ker}(\Phi )\ne 0$. Since $\Phi$ is not trivial, $\Ll$ is a line bundle and $\Ee$ fits into an exact sequence
\begin{equation}\label{eqaa3}
0 \to \Ll \to \Ee \to \Ii _Z\otimes \Aa \to 0
\end{equation}
with $Z$ a zero-dimensional subscheme and $\Aa$ a line bundle. The proof of loc. cit. also gives that $\Phi (\Ee )\subset \Ll \otimes T_X$ and so $\Phi \circ \Phi =0$. Since $\Phi (\Ll )=0 \subset \Ll\otimes T_X$ and $(\Ee ,\Phi ) $ is semistable, we have $\deg \Ll \le \deg \Aa$. Note that $\Phi$ induces a non-zero map $u: \Oo _X\rightarrow \Bb \otimes T_X$ with $\Bb := \Ll \otimes \Aa ^\vee$. In particular, for a fixed integer $m\gg 0$ and a general $C\in |\Oo _X(m)|$, we have $u_{|C} \ne 0$. On the other hand, by \cite[Theorem 4.11]{Pet} and \cite{Miya} the vector bundle $(\Omega ^1_X)_{|C}$ is ample, and so $h^0(C,(\Bb \otimes T_X)_{|C})=0$, a contradiction.
\end{proof}

%%%%%%%%%%%%%%%%%%%%%%%%%
\section{Rank two co-Higgs bundles on surfaces and foliations}
\noindent From now on we always assume that $X$ is a smooth projective surface. 

\begin{remark}
In general, for a co-Higgs bundle $(\Ee, \Phi)$ with a non-injective $\Phi \ne 0$, the sheaf $\Ll  := \mathrm{ker}(\Phi)$ is saturated of rank one in $\Ee$ and so $\Ee$ fits in an exact sequence (\ref{eqaa3}) with $\det (\Ee ) \cong \Ll \otimes \Aa$ and $Z$ a zero-dimensional subscheme. Note that $\Phi$ is nilpotent if and only if $\Phi (\Ee )\subset \Ll\otimes T_X$. Assume that $\Phi$ is nilpotent. From $\Phi (\Ll )=0$, $\Phi $ induces a map $\phi : \Ii _Z\otimes \Aa \rightarrow \Ll \otimes T_X$ with $\mathrm{Im}(\Phi ) = \mathrm{Im}(\phi)$. If $\deg \Ll > \deg \Aa$ (resp. $\deg \Ll \ge \deg \Aa$), then both $\Ee$ and $(\Ee ,\Ll )$ are not semistable (resp. not stable). Now assume $\deg \Ll \le \deg \Aa$ (resp. $\deg \Ll < \deg \Aa$) and that $(\Ee ,\Phi)$ is not stable (resp. semistable). Then there is a saturated line bundle $\Ll' \subset \Ee$ such that $\Phi (\Ll' )\subset \Ll' \otimes T_X$ and $\deg \Ll'  > \deg \Ll$. Since $\Phi (\Ee )\subset \Ll\otimes T_X$, we get $\Ll' \otimes T_X\subseteq \Ll \otimes T_X$, contradicting the inequality $\deg \Ll'  > \deg \Ll$. Hence in case of $\deg \Ll \le \deg \Aa $, we are in the set-up of \cite{Correa}. 
\end{remark}

\begin{remark}
Assume $H^0(T_X)\ne 0$ and take a non-zero section $\sigma \in H^0(T_X)$. For $\Ee := \Oo _X\oplus \Oo _X$ with a fixed basis $\{e_1,e_2\}$ of $H^0(\Ee)$, define $\Phi : \Ee \rightarrow \Ee \otimes T_X$ to be induced by $\begin{pmatrix}
0 & \sigma\\
0 & 0
\end{pmatrix}$. Then $\Phi$ is a non-trivial nilpotent field whose associated foliation is the saturated foliation associated to $\sigma$.
\end{remark}

Let us assume for the moment that $X$ is a smooth and non-rational projective surface $X$ with negative Kodaira dimension $\kappa(X)=-\infty$. Let $u: X\rightarrow  Y$ be the minimal model of $X$ and then $u$ is a finite sequence of blow-ups of points. Denoting the Albanese variety of $X$ by $C$, we get that $C$ is a smooth curve of positive genus $q=h^1(\Oo _X)$ and there is a vector bundle $\Ff$ of rank two on $C$ such that $Y = \PP (\Ff)$. Here we may assume that $\Ff$ is initialized, i.e. $h^0(\Ff)>0$ and $h^0(\Ff \otimes \Ll^\vee)=0$ for all $\Ll \in \mathrm{Pic}^d(C)$ with $d\ge1$. If $\pi :Y\rightarrow C$ denotes the projection, then $f= \pi \circ u: X \rightarrow C$ is the Albanese map of $X$. Letting $T_f$ be the relative tangent sheaf of $f$, we have an exact sequence
$$0 \to T_f\to T_X \to \Ii _W\otimes (\omega _X \otimes T_f)^\vee \to 0$$ 
with $W$ a zero-dimensional subscheme of $X$ supported by the critical locus of $f$.

\begin{proposition}\label{lem4}
Let $(\Ee, \Phi)$ be a semistable co-Higgs bundle of rank two on a surface $X$ with $\Ee$ not semistable and $\Phi\ne 0$. Then we have $\kappa(X)=-\infty$ and $\Phi$ induces a meromorphic foliation on $X$ whose general leaf is a smooth rational curve. Moreover if we assume that $X$ is not rational, then there exists a nonnegative divisor $D$ and a line bundle $\Aa$ such that $\Ee$ fits into an exact sequence
\begin{equation}\label{eqa2}
0 \to T_f(-D)\otimes \Aa \to \Ee \to \Ii _Z\otimes \Aa \to 0
\end{equation}
with $\deg T_f(-D) >0$, where $f: X \rightarrow C$ is the Albanese mapping. 
\end{proposition}

\begin{proof}
Set $\Rr:=\det (\Ee)$. Since $(\Ee ,\Phi)$ is semistable and $\Ee$ is not, then there exists a line subbundle $\Ll\subset \Ee$ such that $\deg \Ll > \deg \Rr/2$ with an exact sequence
\begin{equation}\label{eqar1}
0 \to \Ll\to \Ee \to \Ii _Z\otimes \Rr \otimes \Ll ^\vee \to 0
\end{equation}
with $Z$ a zero-dimensional scheme. The assumption on $(\Ee ,\Phi)$ gives $\Phi (\Ll )\nsubseteq \Ll\otimes T_X$. Composing $\Phi _{|\Ll}$ with the map $\Ee\otimes T_X \rightarrow \Ii _Z\otimes \Rr \otimes \Ll^\vee \otimes T_X$ induced by (\ref{eqar1}), we get a non-zero map $v: \Ll \rightarrow \Ii _Z\otimes \Rr \otimes \Ll^\vee \otimes T_X$ and so an injective map $j: \Ll^{\otimes 2}\otimes \Rr ^\vee \rightarrow T_X$. If we let $\Nn \subset T_X$ be the saturation of $j(\Ll^{\otimes 2}\otimes \Rr ^\vee)$, then we have $\Nn \cong \Ll^{\otimes 2}\otimes \Rr ^\vee (D)$ with either $D=\emptyset$ or $D$ an effective divisor. In particular, we get $\deg \Nn \ge \deg \Ll^{\otimes 2}\otimes \Rr ^\vee  > 0$. By a theorem of Miyaoka and Shepherd-Barron (see \cite{SB}), $\Nn$ is the tangent sheaf of a meromorphic foliation by curves with a rational curve as a general leaf. Since $q$ is positive, this foliation is induced by the Albanese map $f: X\rightarrow C$, i.e. $\Nn\cong T_f$, and we may take $\Aa := \Rr \otimes \Ll^\vee$.  
\end{proof}

\section{Surfaces with $H^0(T_X)=H^0(S^2T_X)=0$}

Recall that if there is no non-trivial global tangent vector field, then the existence of unstable co-Higgs bundle of rank two implies that the co-Higgs field is nilpotent by Lemma \ref{tt6}. On the other hand, it is observed in Lemma \ref{oo1} that any non-trivial global tangent vector field suggests an example of strictly semistable co-Higgs bundle of arbitrary rank with injective co-Higgs field. 

\begin{lemma}[\cite{R1, Correa}]\label{c4}
Let $X$ be a smooth surface such that $H^0(T_X)=H^0(S^2T_X)=0$. If $(\Ee, \Phi)$ is any co-Higgs bundle of rank two on $X$, then $\Phi$ is nilpotent. Moreover, if $(\Ee, \Phi)$ is stable and $\kappa (X)\ge0$, then we have $\Phi=0$. 
\end{lemma}

\begin{proof}
Note that $H^0(T_X)=0$ implies that $\Phi$ is trace-free. The first assertion is from the proof of \cite[Theorem 7.1, page 148--149]{R1}. For the second, if $\Phi \ne 0$, then by \cite[Theorem 1.1]{Correa} $(\Ee, \Phi)$ is strictly semistable, a contradiction. 
\end{proof}

In \cite{Correa} the classification of smooth surfaces with semistable co-Higgs bundles of rank two with nilpotent co-Higgs fields is done. This together with Lemma \ref{c4} motivates to investigate the criterion for the vanishing $H^0(T_X)=H^0(S^2T_X)=0$. 

\begin{remark}\label{aaa0}
Let $X$ be a smooth projective surface. The space $H^0(T_X)$ of global tangent vector fields, is the tangent space of the algebraic group $\mathrm{Aut}^0(X)$ at the identity map $X\rightarrow X$. So if $\pi : X \rightarrow Y$ is the blow-up at $p\in Y$, then we get $H^0(X,T_X) \cong H^0(Y,\Ii _p\otimes T_Y)$. 
\end{remark}

\noindent Iterating the observation in Remark \ref{aaa0}, we get information on $H^0(T_X)$, with respect to $\kappa (X)$. Recall that any smooth rational surface has as its minimal model either $\PP^2$ or a Hirzebruch surface $\FF_e$ for $e\in \NN \setminus \{1\}$. 

\begin{enumerate}
\item Assume that $X$ is rational and that there is a birational morphism $X\rightarrow Y$ with $Y$ the blow-up of $\PP^2$ at four points of $\PP^2$, no three of them collinear. Then we get $H^0(T_X)=0$.

\item Let $\pi _i: \FF_0=\PP^1\times \PP^1 \rightarrow \PP^1$ for $i=1,2$, denote the projection on the $i$-th factor. Let $Y$ be the blow-up of $\FF_0$ at a finite set of points $S\subset \FF_0$ such that $\sharp (\pi _i(S))\ge 3$ for each $i$. Then we get $H^0(T_Y)=0$. Thus if there is a birational morphism $X\rightarrow Y$, then we also get $H^0(T_X)=0$.

\item We have $\FF_e = \PP (\Oo _{\PP^1}\oplus \Oo _{\PP^1}(-e))$ and every automorphism of $\FF_e$ preserves the ruling $\PP (\Oo _{\PP^1}\oplus \Oo _{\PP^1}(-e)) \rightarrow \PP^1$. Let us take $e>0$. If $S$ is a general subset of $\FF_e$ with $\sharp (S) \ge \max \{3,e+1\}$, then we get $h^0(\Ii _{S}\otimes T_{\FF_e})=0$ from $h^0(\mathcal{E}nd (\Oo _{\PP^1}\oplus \Oo _{\PP^1}(-e)) = 3+e$. If we assume the existence of a birational morphism $X \rightarrow Y$ with $Y$ the blow-up of $\FF_e$ at $S$, then we get $h^0(T_X)\le h^0(T_Y)=0$.

\item Let us assume $X$ is a non-rational surface with $\kappa(X)=-\infty$. Then we have $q:= h^1(\Oo _X)>0$ and the Albanese variety $f:X\rightarrow C$ of $X$ has as target a smooth curve $C$ of genus $q$ and $\PP^1$ as its general fiber. There exists a vector bundle $\Gg$ of rank two on $C$ and a finite sequence $u: X\rightarrow Y$ of blowing ups such that $\pi : Y = \PP(\Gg )\rightarrow C$ is $\PP^1$-bundle over $C$ and $f=\pi \circ u$. Here we may assume that $\Gg$ is initialized, i.e. $h^0(\Gg)>0$ and $h^0(\Gg \otimes \Ll^\vee)=0$ for all $\Ll \in \mathrm{Pic}^d(C)$ with $d\ge1$. If $q\ge 2$ and $\Gg$ is simple, then we get $H^0(T_Y)=0$ and so $H^0(T_X)=0$. If $\Gg$ is not simple, then we get an upper bound for $h^0(T_Y)$ in terms of $h^0(\mathcal{E}nd (\Gg ))$ and so, as in part ({c}), we get many $X$ with minimal model $Y$ and $h^0(T_X)=0$. This observation also applies in the case $q=1$, when $\Gg$ is decomposable. If $q=1$ and $\Gg$ is indecomposable, then we get $h^0(\mathcal{E}nd (\Gg ))=2$ by Atiyah's classification of vector bundles on elliptic curves.

\item Assume $\kappa (X) =0$ and let $u: X \rightarrow Y$ be the morphism to its minimal model. From the list of possible surface $Y$ in \cite[page 244]{BHPV} we have $H^0(T_Y) =0$, unless $Y$ is an Abelian surface or it has an Abelian surface as a finite unramified covering. If $Y$ is an Abelian surface, then we have $H^0(\Ii _p\otimes T_Y)=0$ for all $p\in Y$. Hence we get $H^0(T_X)\ne 0$ if and only if $X$ is either an Abelian surface or a bi-elliptic surface. In particular, we have $H^0(T_X)=0$ if $X$ is not a minimal model.

\item Let $Y$ be a projective surface with $\kappa (Y)=1$ and $h^0(T_Y)>0$. Since $Y$ has only finitely many rational curves and only a one-dimensional family of elliptic curves, we have $h^0(T_Y) =1$ and so $h^0(\Ii _p\otimes T_Y)=0$ for a general $p\in Y$. Thus we have $H^0(T_X)=0$ if $X$ has $Y$ as its minimal model, but it factors through a blow-up of $Y$ at a general point of $Y$. We saw that $H^0(T_X)=0$ if the minimal model $Y'$ of $X$ satisfies $H^0(T_{Y'})=0$.

\item In case of $\kappa (X)=2$, we have $H^0(T_X) =0$, because $X$ has neither a positive dimensional family of elliptic curves nor a positive dimensional family of rational curves.
\end{enumerate}

\begin{remark}
Let $X$ be a smooth projective surface. Assume the existence of a finite unramified covering $u: Y\rightarrow X$ with $H^0(S^2T_Y) =0$. Since $u$ is unramified, we get $u^\ast (S^2T_X) =S^2T_Y$ and so $H^0(S^2T_X)=0$. In particular \cite[Lemma 7.1]{R1} shows that $H^0(S^2T_X)=0$, if $X$ is an Enriques surface.
\end{remark}

Let $u: X\rightarrow Y$ be the blow-up at one point $p\in Y$ and set $D:= u^{-1}({p})$. We have $D\cong \PP^1$ and $\Oo _D(D)$ is the line bundle on $D$ of degree $-1$. The natural map $H^0(\Omega ^1_Y) \rightarrow H^0(\Omega ^1_X)$ is an isomorphism. Hence $H^0(S^2\Omega ^1_X)$  contains a three-dimensional linear subspace spanning $S^2\Omega ^1_X$ outside the exceptional locus $D$ of the map $u: X\rightarrow Y$. We have an exact sequence
\begin{equation}\label{equ1}
0 \to T_X \to u^\ast (T_Y) \to \Oo _D(-D)\to 0
\end{equation}
(see \cite[Lemma 15.4 (iv)]{Fu}), in which $\Oo _D(-D)$ is the line bundle on $D \cong \PP^1$ of degree $1$. From (\ref{equ1}) we also get a map $S^2T_X \rightarrow u^\ast (S^2T_Y)$, which is an isomorphism outside $D$. Applying $u_\ast$, we get a map $u_\ast (S^2T_X) \rightarrow S^2T_Y$ which is an isomorphism on $Y\setminus \{p\}$, and in particular it is injective as a map of sheaves.

\begin{remark}\label{aaa2}
Let $u: X\rightarrow Y$ be a birational morphism of smooth surfaces. Since $u$ is a finite sequence of blow-ups of points, we have $u_\ast (\Oo _X)=\Oo _Y$ and $R^iu_\ast (\Oo _X)=0$ for all $i>0$. The natural map $S^2T_X \rightarrow u^\ast (S^2T_Y)$ is injective, because it is an isomorphism outside finitely many divisors of $X$. Applying $u_\ast$ and the projection formula, we get an inclusion $u_\ast (H^0(S^2T_X)) \subseteq H^0(S^2T_Y)$.
\end{remark}

\begin{lemma}\label{aaa1}
Let $u: X \rightarrow Y$ be a blow-up at a point $p$ on a smooth surface $Y$. If $j: H^0(S^2T_X) \rightarrow H^0(S^2T_Y)$ is the map induced by $u_\ast$, then $\mathrm{Im}(j)\subseteq H^0(\Ii _p\otimes S^2T_Y)$.
\end{lemma}

\begin{proof}
It is sufficient to prove that $\mathrm{Im}(j_1)\subseteq H^0(\Ii _p\otimes T_Y^{\otimes2})$, where $j_1: H^0(T_X^{\otimes 2}) \to H^0(T_Y^{\otimes 2})$ is the map induced by $u_\ast$. The corresponding map $u_\ast (T_X) \rightarrow T_Y$ has $\Ii _p\otimes T_Y$ as its image
and $u_\ast (\Oo _D(-D)) =\CC^{\oplus 2}_p$. The sequence (\ref{equ1}) tensored with $u^\ast (T_Y)$ gives an inclusion $T_X\otimes u^\ast (T_Y)\rightarrow u^\ast (T_Y^{\otimes 2})$.
Applying $u_\ast$ and the projection formula, we get an injective map $j_2: H^0(\Ii _p\otimes T_Y^{\otimes 2})\rightarrow H^0(T_Y^{\otimes 2})$. Now $j_1$ factors through $j_2$, implying the assertion.
\end{proof}

\begin{lemma}
Let $Y$ be a smooth projective surface such that $H^0(S^2T_Y)=0$. For a birational morphism $u: X\rightarrow Y$, we get $H^0(S^2T_X)=0$.
\end{lemma}

\begin{proof}
Since $u$ is the composition of finitely many blow-ups, it is sufficient to prove it when $u$ is the blow-up at a point $p\in Y$. Denote by $E:= u^{-1}({p})$ the exceptional divisor and set $s':=s_{|X\setminus E}$ for a section $s\in H^0(S^2T_X)$. Then $s'$ induces $s_1\in H^0(Y\setminus \{p\},(S^2T_Y)_{|Y\setminus \{p\}})$ and it extends to $\sigma \in H^0(S^2T_Y)$ by Hartogs' theorem, because $S^2T_Y$ is locally free. By assumption we have $\sigma =0$. It implies that $s'=0$, which also implies $s=0$.
\end{proof}

\begin{lemma}\label{u3}
Let $X$ be a smooth and compact complex surface whose minimal model $Y\not\cong X$ is either a complex torus or $C\times W$ with $C$ an elliptic curve and $W$ a smooth curve of genus at least two. Then we have $H^0(S^2T_X)=0$.
\end{lemma}

\begin{proof}
Since $X\not \cong Y$, there is a blow-up $B\rightarrow Y$ at a point $p\in Y$ and a birational morphism $X\rightarrow B$; Remark \ref{aaa0} and Lemma \ref{aaa1} give $H^0(T_B)=H^0(S^2T_B)=0$ and then Remark \ref{aaa2} gives $H^0(T_X)=H^0(S^2T_X)=0$. 
\end{proof}

\begin{remark}\label{bb2}
Let $f: X\rightarrow W$ be an unramified covering between smooth surfaces. If $H^0(T_W)\ne 0$ (resp. $H^0(S^2T_W) \ne 0$), then we have $H^0(T_X)\ne 0$ (resp. $H^0(S^2T_W) \ne 0$). Now assume that $W$ is not minimal and take any $D\subset W$ with $D\cong \PP^1$ and $D^2=-1$. Let $v: W\rightarrow W'$ be the blow-down of $D$. Since $D$ has an open simply connected neighborhood in the Euclidean topology, there is an unramified covering $f': X'\rightarrow W'$ and a map $v': X\rightarrow  X'$ such that $v'$ is the blow-down of $\deg (f)$ disjoint exceptional curves and $f'\circ v' = v\circ f$.
\end{remark}

\begin{proposition}\label{bb1}
Let $Y$ be a smooth surface and $S\subset Y$ be any finite subset such that $H^0(\Ii _S\otimes T_Y)=H^0(\Ii _S\otimes S^2T_Y)=0$. If $X$ is any smooth surface with a birational morphism $u: X\rightarrow Y$ with $S$ contained in the image of divisors of $X$ contracted by $u$, then we have $h^0(T_X)=h^0(S^2T_X)=0$. In particular, $\Phi$ is trivial for any co-Higgs bundle of rank two on $X$. 
\end{proposition}

\begin{proof}
By Remarks \ref{aaa0}, \ref{aaa2} and Lemma \ref{aaa1}, we have $H^0(T_X)=H^0(S^2T_X)=0$. Now we may apply Lemma \ref{c4} for the second assertion. 
\end{proof}

\begin{remark}
Let $Y$ be a minimal surface with $\kappa (Y)=0$. We have $H^0(T_Y)=H^0(S^2T_Y)=0$, unless either $Y$ is an Abelian surface or a bielliptic surface. Thus in Proposition \ref{bb1} with $\kappa (Y)=0$, we may take either $S=\emptyset$ or, in the two exceptional cases, as $S$ any point of $Y$ by Remark \ref{bb2}. 
\end{remark}

\begin{remark}
The surface $Y =C\times W$ described in Lemma \ref{u3} and Proposition \ref{v3}  is the only one, up to a finite unramified covering, with $H^0(T_Y)\ne 0$ and $\kappa (Y) =1$. This can be obtained by Lemma \ref{aaa1} and Remark \ref{bb2} for these surfaces and their unramified coverings.
\end{remark}

\begin{proposition}\label{v2}
Let $(\Ee, \Phi)$ be a co-Higgs bundle of rank two on an Abelian surface $X$. 
\begin{itemize}
\item [(i)] If $\Ee$ is simple, then either $\Phi =0$ or $\Phi$ is injective. In particular, $(\Ee ,\Phi )$ is stable, semistable, strictly semistable or unstable if and only if $\Ee$ has the same property.
\item [(ii)] If $\Ee$ is not semistable, then $(\Ee ,\Phi)$ is not semistable.
\item [(iii)] If $\Ee$ is strictly semistable and indecomposable, but not simple, then $(\Ee ,\Phi)$ is strictly semistable.
\end{itemize}
\end{proposition}

\begin{proof}
Note that we have $T_X\cong \Oo _X^{\oplus 2}$ and so $\Phi : \Ee \rightarrow \Ee ^{\oplus 2}$. If $\Ee$ is simple, then there are $c_1,c_2\in \CC$ such that $\Phi = (c_1\id_{\Ee},c_2\id_{\Ee})$. So
for any sheaf $\Bb \subset \Ee$, we have $\Phi (\Bb)\subset \Bb \otimes T_X$, implying (i). The assertion (ii) is a special case of Theorem \ref{t2.10}. 

%Assume that $\Ee$ is not semistable and then there is an exact sequence 
%\begin{equation}\label{eqa+1}
%0 \to \Ll \to \Ee \to \Ii _Z\otimes \Aa \to 0
%\end{equation}
%with $Z\subset X$ a zero-dimensional scheme, $\Ll$ and $\Aa$ line bundles and $\deg \Ll  >\deg \Aa$. So the composition of $\Phi$ with the map $\Ee \otimes T_X\rightarrow \Ii _Z\otimes \Aa \otimes T_X$ induced by (\ref{eqa+1}) is the zero-map. Thus we get $\Phi (\Ll )\subset \Ll \otimes T_X$, proving that  $(\Ee ,\Phi )$ is not semistable.

Now assume that $\Ee$ is strictly semistable, but not simple. Since $\Ee$ is strictly semistable, then it fits in an exact sequence (\ref{eqaa1}) with $\deg \Aa=\deg \Ll$. Since $\Ee$ is not simple, either (\ref{eqaa1}) is not the unique destabilizing sequence of $\Ee$ or $\Hom (\Ii _Z\otimes \Aa ,\Ll )$ is not trivial. In the former case we have $Z=\emptyset$ and $\Ee \cong \Ll\oplus \Aa$; in this case we get
$$ \dim \End (\Ee)=\left\{
                                           \begin{array}{ll}
                                             2, & \hbox{if $\Aa \not \cong \Ll$;}\\                                      
                                             4, & \hbox{if $\Aa \cong \Ll$}
                                            \end{array}
                                         \right.$$
In the latter case the condition $\deg \Aa =\deg \Ll$ implies $\Aa \cong \Ll$. So we get either either $Z=\emptyset$ and (\ref{eqaa1}) splits or $\dim \End (\Ee ) =3$. If $\Ee$ is indecomposable, then we have $\Aa \cong \Ll$ and any endomorphism of $\Ee$ sends $\Ll$ into itself. Thus we get $\Phi (\Ll )\subset \Ll\otimes T_X$.
\end{proof}

\begin{proposition}\label{v3}
Let $X =C\times W$, where $C$ and $W$ are smooth curves of genus one and $g\ge 2$, respectively.  
\begin{itemize}
\item [(i)] If $\Ee$ is simple, then either $\Phi =0$ or $\Phi$ is injective. In particular, $(\Ee ,\Phi )$ is stable, semistable, strictly semistable or unstable if and only if $\Ee$ has the same property.
\item [(ii)] If $\Ee$ is not semistable, then $(\Ee ,\Phi)$ is not semistable.
\item [(iii)] If $\Ee$ is strictly semistable and indecomposable, but not simple, then $(\Ee ,\Phi)$ is strictly semistable.
\end{itemize}
\end{proposition}

\begin{proof}
Let $\pi _2: X\rightarrow W$ denote the projection and then we have $T_X\cong \Oo _X \oplus \Rr$ with $\Rr^\vee \cong  \pi _2^\ast (\omega _W)$. Here $\Rr^\vee$ is spanned, but not trivial. If $\Ee$ is either simple or semistable, then $\Phi$ is associated to an endomorphism of $\Ee$, because there is no non-zero map $\Ee \rightarrow \Ee \otimes \Rr $. If $\Ee$ is simple, then we get $\Phi (\Bb )\subset \Bb \otimes \Oo _X\oplus \{0\}\subset \Bb \otimes T_X$ and so the part (i). Parts (ii) and (iii) are proved as in Proposition \ref{v2}.
\end{proof}

%\begin{proposition}
%Let $X$ be a smooth surface with $\kappa (X) \ge 0$ and $H^0(T_X)=H^0(S^2T_X)=0$. If a co-Higgs bundle $(\Ee ,\Phi )$ of rank two on $X$ is semistable, then $\Ee$ is also semistable. 
%\end{proposition}

%\begin{proof} 
%We may assume that $\Phi \ne 0$. By the proof of \cite[Theorem 7.1, page 148--149]{R1}, $\Phi$ is nilpotent and so by \cite[Theorem 1.1]{Correa} $X$ is a minimal model with an unramified covering $v: X'\rightarrow X$, where $X'$ is either 
%\begin{itemize}
%\item an Abelian surface, or
%\item $C\times W$ with $C$ an elliptic curve and $W$ a smooth curve of genus $g \ge 2$. 
%\end{itemize}
%Set $\Ee ':= v^\ast (\Ee)$ and use $v^\ast (\Oo _X(1))$ as a polarization on $X'$. Since $v$ is unramified, $\Phi ':= v^\ast (\Phi)$ is a co-Higgs field. By \cite[Lemma 3.2.2]{HL} $\Ee$ is slope-semistable if and only if $\Ee '$ is slope-semistable. Assume that $\Ee$ is not semistable with an exact sequence (\ref{eqaa1}), where $\Ll, \Aa \in \Pic (X)$ with $\deg \Ll > \deg \Aa$. From this we get an exact sequence;
%\begin{equation}\label{eqaaaa2}
%0 \to \Ll ' \to \Ee ' \to \Ii _{Z'}\otimes \Aa '\to 0
%\end{equation}
%with $\Ll ' =v^\ast (\Ll)$, $\Aa ' =v^\ast (\Aa)$ and $Z' = v^{-1}(Z)$. The proofs of Propositions  \ref{v2} and \ref{v3} give that $\Phi '(\Ll )\subset \Ll '\otimes T_{X'}$, and so $\Phi (\Ll)\subset \Ll\otimes T_X$. Thus $(\Ee ,\Phi)$ is not semistable.
%\end{proof}

\begin{remark}\label{c1}
Let $\FF_e$ for $e\ge 0$, be a Hirzebruch surface with $\mathrm{Pic}(\FF_e)\cong \ZZ^{\oplus 2} \cong \ZZ\langle h,f\rangle$, where $f$ is a ruling $\pi : \FF_e\rightarrow \PP^1$ and $h$ is a section of this ruling with $h^2=-e$. We have 
\begin{align*}
T_{\FF_e} &\cong \Oo _{\FF_e}(2f)\oplus \Oo _{\FF_e}(2h+ef),\\
S^2T_{\FF_e} &\cong \Oo _{\FF_e}(4f) \oplus \Oo _{\FF_e}(2h+(e+2)f)\oplus \Oo _{\FF_e}(4h+2ef)).
\end{align*}
Now fix a finite subset $S\subset \FF_e$. 

(a) Assume $e=0$ and we may consider $\FF_0$ as a smooth quadric surface in $\PP^3$. Letting $\eta : \FF_0 \rightarrow \PP^1$ be the other projection, we have $h^0(\Ii _S\otimes T_{\FF_0}) =0$ if and only if $\sharp (\pi (S)) \ge 3$ and $\sharp (\eta (S)) \ge 3$. We have $h^0(\Ii _S\otimes S^2T_{\FF_0}) =0$ if and only if $\sharp (\pi (S)) \ge 5$, $\sharp (\eta (S)) \ge 5$ and $S$ is not contained in other quadric surfaces in $\PP^3$. If $h^0(\Ii _S\otimes S^2T_{\FF_0}) =0$, then we have $\sharp (S) \ge 9$. Conversely, if $\sharp (S)\ge 9$ and $S$ is general in $\FF_0$, then we get $h^0(\Ii _S\otimes S^2T_{\FF_0}) =0$

(b) Assume $e=1$ and let $u: \FF_1\rightarrow \PP^2$ denote the blow-down of $h$. The linear system $|\Oo _{\FF_1}(2h+f)|$ has $h$ as its base locus and $u$ is induced by $|h+f|$. Then the linear system $|\Oo _{\FF_1}(4h+2f)|$ has $2h$ as its base locus and $H^0(\Oo _{\FF_1}(2h+2f)) = u^\ast (H^0(\Oo _{\PP^2}(2))$. We have $h^0(\Ii _S\otimes T_{\FF_1}) =0$ if and only if $\sharp (\pi (S)) \ge 3$. We have $H^0(\Ii _S\otimes S^2T_{\FF_1}) =0$ if and only if $h^0(\PP^2,\Ii _{u(S)}(2)) =0$.

(c) Assume $e\ge 2$ and that $S$ is general. We have 
\begin{align*}
&h^0(\Oo _{\FF_e}(2h+ef)) = h^0(\Oo _{\FF_e}(h+ef)) = e+2, \\
&h^0(\Oo _{\FF_e}(4h+2ef)) = h^0(\Oo _{\FF_e}(2h+2ef)) = 3e+3,
\end{align*}
and it implies that $H^0(\Ii _S\otimes T_{\FF_e}) =0$ if $\sharp (S) \ge e+2$ and $H^0(\Ii _S\otimes S^2T_{\FF_e})=0$ if $\sharp (S) \ge 3e+3$.
\end{remark}

\begin{remark}\label{c2}
Take a finite subset $S\subset \PP^2$ and assume the existence of $p\in S$ such that $h^0(\Ii _{S\setminus \{p\}}(2)) =0$. By part (b) of Remark \ref{c1} we have $h^0(\Ii _S\otimes T_{\PP^2}) =h^0(\Ii _S\otimes S^2T_{\PP^2})=0$.
\end{remark}

Lemma \ref{c4} with Remarks \ref{c1} and \ref{c2} gives the following.

\begin{corollary}\label{c3}
Let $X$ be a rational surface with a relative minimal model $u: X\rightarrow Y$ and fix a finite subset $S\subset Y$ such that $u^{-1}(o)$ contains a curve for $o\in S$. Assume that 
\begin{itemize}
\item if $Y =\PP^2$, then there exists $p\in S$ such that $H^0(\Ii _{S\setminus \{p\}}(2)) =0$; 

\item if $Y =\FF_e$, then $S$ is general in $\FF_e$ and $\sharp (S)\ge \left\{
                                           \begin{array}{ll}
                                             9-3e, & \hbox{if $e\le 1$;}\\                                      
                                             3e+3, & \hbox{if $e\ge 2$.}
                                            \end{array}
                                         \right.$
\end{itemize}
Then $\Phi$ is nilpotent for any co-Higgs bundle $(\Ee, \Phi)$ of rank two on $X$.
\end{corollary}

%%%%%%%%%%%%%%%%%%%%%%%%%%%

\providecommand{\bysame}{\leavevmode\hbox to3em{\hrulefill}\thinspace}
\providecommand{\MR}{\relax\ifhmode\unskip\space\fi MR }
\providecommand{\MRhref}[2]{%
  \href{http://www.ams.org/mathscinet-getitem?mr=#1}{#2}
}
\providecommand{\href}[2]{#2}

\end{document}